\documentclass[11pt,twoside]{amsart}
\usepackage{latexsym}
\usepackage{graphics}
\usepackage[pdftex]{graphicx}
\usepackage{multirow}
\usepackage{caption}
\usepackage{subcaption}
\usepackage{amsmath, amsthm, amscd, amsfonts, amssymb, float, graphicx, color, xcolor, soul}
\input xy
\xyoption{all}
\newtheorem{lemma}{Lemma}
\newtheorem{theorem}[lemma]{Theorem}

\newtheorem{conj}{Conjecture}
\newtheorem*{theorem*}{Theorem}
\newtheorem*{conj*}{Conjecture}

\begin{document}
%
\title[On harmonic index and diameter of quasi-tree graphs]
{On harmonic index and diameter of quasi-tree graphs}
\author[A.~Abdolghafourian]{A.~Abdolghafourian}
\address{A.~Abdolghafourian, Department of Mathematical Science \newline Yazd University
\\ Yazd, 89195-741, Iran}
\email{abdolghafourian@yazd.ac.ir}
\author[M.~A.~Iranmanesh]{Mohammad~A.~Iranmanesh*}
\thanks{*Corresponding author}\address{M.~A.~Iranmanesh, Department
of Mathematical Science \newline Yazd University\\ Yazd, 89195-741 , Iran}
\email{iranmanesh@yazd.ac.ir}

\begin{abstract}
The harmonic index of a graph $G$ ($H(G)$) is defined as the sum of the weights $\frac{2}{d_u+d_v}$ for all edges $uv$
of $G$, where $d_u$ is the degree of a vertex $u$ in $G$. In this paper, we show that $H(G)\geq D(G)+\frac{5}{3}-\frac{n}{2}$
and $H(G)\geq \left(\frac{1}{2}+\frac{2}{3(n-2)}\right)D(G)$ where $G$ is a quasi-tree graph of
order $n$ and diameter $D(G)$. This is a conjecture, proposed by Jerline and Michaelraj \cite{amalorpava2016harmonic}. Indeed we show
that both lower bounds are tight and identify all quasi-tree graphs reaching these two lower bounds.
\end{abstract}
\maketitle
\section{Introduction}
\label{sec:introd}
Let $G$ be a simple connected graph with vertex set $V(G)$ and  edge set $E(G)$ of order $n$ ($|V(G)|=n$). The harmonic
index of $G$, first appeared in \cite{fajtlowicz1987conjectures}, is defined as $H(G)=\sum_{uv\in E(G)}\frac{2}{d_u+d_v}$,
where for $v\in V(G)$, $d_v$ is the degree of $v$ in $G$. For $u,  v\in V(G)$, the distance between $u$ and $v$ is shown by $d(u,v)$. Also $D(G)=max\{d(u,v)\}_{u,v\in V(G)}$ is the diameter of $G$ and
$\delta(G)=min \{d_v\} _{v\in V(G)}$.

The applications of harmonic index in various chemical disciplines have been demonstrated in  \cite{betancur2015vertex,deng2013harmonic,furtula2013structure,rada2014vertex}.
Also several studies have focused on  graph theoretical properties of the harmonic index.
For a broad overview, we refer to \cite{gutman2019harmonic}.

A connected graph $G$ is a quasi-tree graph if $G$ is not a tree and there exists a
vertex $v\in V(G)$ such that $G-v$ is a tree. A graph $G$ is called unicyclic, if it contains only one cycle.
Obviously, every unicyclic graph is a quasi-tree graph. Many researchers have studied  topological indices of quasi-tree graphs.
See for example \cite{iranmanesh2020estrada,li2009extremal,qiao2010zagreb,qiao2010zeroth,wagner2010maxima,xu2015kirchhoff}.

Liu \cite{liu2013harmonic} found a relation between harmonic index and diameter of a graph.
He proved that if $n\geq 4$ and $G$ is a connected graph of order $n$,
then $ H(G)\leq D(G)+\frac{n}{2}-1$ and $ H(G)\leq \frac{n}{2}D(G)$.
Also a lower bound was found for trees. If $T$ is a tree of order $n\geq 4$,
then $ H(T)\geq D(T)+\frac{5}{6}-\frac{n}{2}$ and $H(T)\geq \left(\frac{1}{2}+\frac{1}{3(n-1)}\right)D(T)$.
Thereby Liu  \cite{liu2013harmonic} proposed the following conjecture:
\begin{conj}\label{conj1}
Let $G$ be a connected graph with order $n\geq  4$, then
$$H(G)\geq D(G)+\frac{5}{6}-\frac{n}{2}\hspace{.6cm} and \hspace{.6cm} H(G)\geq \left(\frac{1}{2}+\frac{1}{3(n-1)}\right)D(G).$$
\end{conj}
Jerline and Michaelraj  \cite{amalorpava2016harmonic,jerline2016conjecture} found a sharper bound for unicyclic graphs.
They showed that if $G$ is a unicyclic graph of order $n$, then $ H(G)\geq D(G)+\frac{5}{3}-\frac{n}{2}$
and $ H(G)\geq \left(\frac{1}{2}+\frac{2}{3(n-2)}\right)D(G)$. They introduced a family of graphs,
$U_{n,4}^{1,n-5}$, which is a set of graphs obtained from $C_4$ by attaching one pendant vertex and a path of length $n-5$ to two
diametrically nonadjacent vertices of $C_4$ (see Figure \ref{fig:gama}). Then, they proposed the following conjecture  \cite{amalorpava2016harmonic}:
\begin{conj}\label{conj}
Let $G$ be a simple connected graph, that is not a tree, of order
$n\geq 7$. Then,
$$H(G)\geq D(G)+\frac{5}{3}-\frac{n}{2}\hspace{.6cm} and \hspace{.6cm}H(G)\geq \left(\frac{1}{2}+\frac{2}{3(n-2)}\right)D(G),$$
where equality holds if and only if $G = U_{n,4}^{1,n-5}$.
\end{conj}
%

They also show  that the inequalities of the above conjecture are not true for a graph of order $6$, namely
$U_{6,4}^{1,1}$.

\begin{figure}[ht]
\includegraphics[scale=.7]{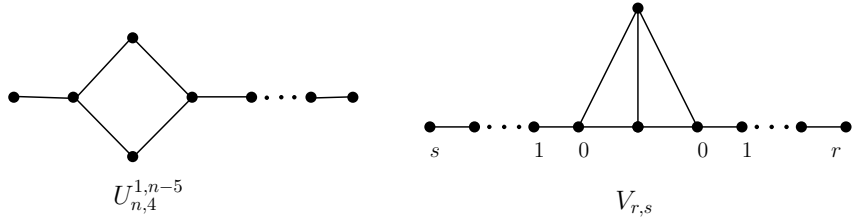}
\caption{The graphs $U_{n,4}^{1,n-5}$ and $V_{r,s}$. }\label{fig:gama}
\end{figure}


Suppose $K_4^{-}$ is a graph of order $4$ which is obtained from $K_4$ by deleting an edge.
Also for $r,s\geq 0$, let $V_{r,s}$ be a family of graphs obtained from $K_4^{-}$ by attaching two paths of lengths $r$ and $s$ to two
nonadjacent vertices of $K_4^{-}$ (see Figure \ref{fig:gama}).
We will show that the inequality holds for all quasi-tree graphs except $U_{6,4}^{1,1}$ and the graph $U_{5,3}^{1,1}$ which is obtained by attaching two pendant vertices to two vertices of $K_3$. Also the equality holds for $V_{1,1}$.

Two main theorems of this paper are as follows.
\begin{theorem*}
Let $G\neq U_{6,4}^{1,1},~U_{5,3}^{1,1}$ be a quasi-tree graph of order $n\geq 3$. Then,
$H(G)\geq D(G)+\frac{5}{3}-\frac{n}{2}$. The equality holds if and only if $G=V_{1,1}$ or $U_{n,4}^{1,n-5}$.
\end{theorem*}
\begin{theorem*}
Let $G$ be a quasi-tree graph of order $n\geq 3$  and $G\neq U_{6,4}^{1,1},~U_{5,3}^{1,1}$. Then,
$H(G)\geq \left(\frac{1}{2}+\frac{2}{3(n-2)}\right)D(G)$. The equality holds if and only if $G=V_{1,1}$ or $U_{n,4}^{1,n-5}$.
\end{theorem*}
In Section \ref{sec:preli}, we prove the lemmas that will be used in Section \ref{sec:main},  where we prove the main theorems.

All graphs considered in this paper are finite, undirected, connected and simple. Let $G$ be a graph and $v\in V(G)$ and $P$ a path of $G$, then by $G-v$ and $G-P$ we mean the graph obtained from $G$ by deleting the vertex $v$ and the vertices of $P$, respectively. For all other notation and
definitions not given here, the readers are referred to \cite{west2001introduction}.
\section{Preliminaries}
\label{sec:preli}
\begin{lemma}\label{case2.2}
For $x,y\geq 2$, the two-variables function
$$f(x,y)=\frac{x+4}{x(x+1)(2+x)}+\frac{y+4}{y(y+1)(2+y)}-\frac{2}{(x+y)(x+y-2)}$$
is positive.
\end{lemma}
\begin{proof}
\begin{align*}
f(x,y)&=\frac{x+4}{x(x+1)(2+x)}+\frac{y+4}{y(y+1)(2+y)}-\frac{2}{(x+y)(x+y-2)}\\
&=\frac{2}{x(x+1)}+\frac{-1}{(x+2)(x+1)}+\frac{2}{y(y+1)}+\frac{-1}{(y+2)(y+1)}-\frac{2}{(x+y)(x+y-2)}\\
&=(\frac{1}{x(x+1)}-\frac{1}{(x+2)(x+1)})+(\frac{1}{y(y+1)}-\frac{1}{(y+2)(y+1)})\\
&+(\frac{1}{x(x+1)}-\frac{1}{(x+y)(x+y-2)})+(\frac{1}{y(y+1)}-\frac{1}{(x+y)(x+y-2)}).
\end{align*}
Given $x,y\geq 2$, the terms inside parentheses are positive.
\end{proof}

\begin{lemma}\label{casen6}
$\frac{x}{x+2}\geq\frac{1}{5+x}+\frac{x-1}{2+x}\geq \frac{11}{28}$ for every $x\geq 2$.
\end{lemma}
\begin{proof}
The first inequality is valid since
$\frac{x}{x+2}=\frac{1}{x+2}+\frac{x-1}{x+2}\geq\frac{1}{5+x}+\frac{x-1}{2+x}$.
\\
Let $f(x)=\frac{1}{5+x}+\frac{x-1}{2+x}$. Then, $f'(x)=\frac{-1}{(5+x)^2}+\frac{3}{(2+x)^2}>0$.
So, $f$ is an increasing function and $f(x)\geq f(2)=\frac{11}{28}$ for every $x\geq 2$.
\end{proof}
\section{Proof of the main theorems}\label{sec:main}
In this section, we will show that the Conjecture \ref{conj} is true for all quasi-tree graphs.
Also in our proof, it will be shown that the equality in both inequalities hold whenever $G$ is the graph $V_{1,1}$.

\begin{figure}[ht]
\includegraphics[scale=.85]{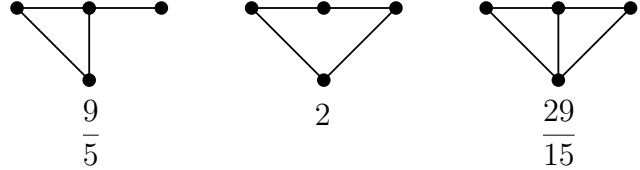}
\caption{Quasi-tree graphs of order $4$ and their harmonic indices.}\label{fign4}
\end{figure}

\begin{figure}[ht]
\includegraphics[scale=.85]{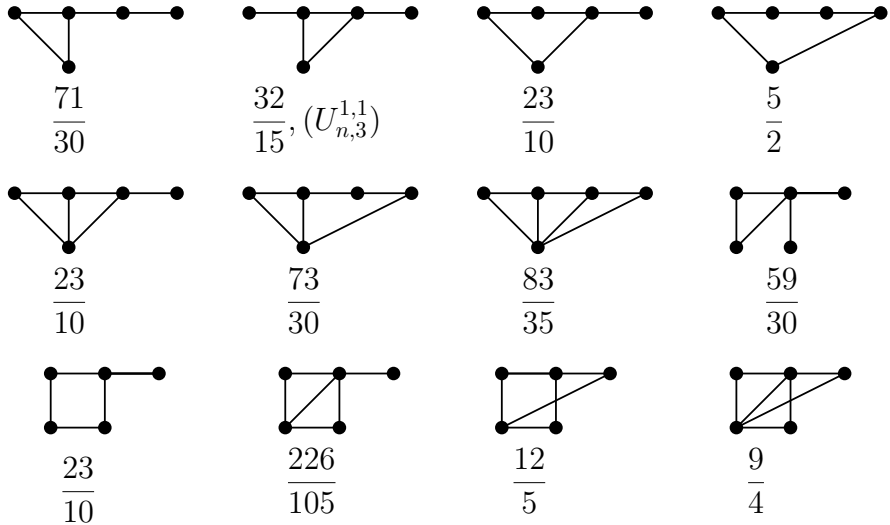}
\caption{Quasi-tree graphs of order $5$ and their harmonic indices.}\label{fign5}
\end{figure}

\begin{lemma}\label{th0}
Let $G$ be a quasi-tree graph of order $n$, where $3\leq n\leq 6$, such that $G\neq U_{6,4}^{1,1},~U_{5,3}^{1,1}$. Then,
$$(1)~H(G)\geq D(G)+\frac{5}{3}-\frac{n}{2}\hspace{.7cm}and\hspace{.7cm}(2)~H(G)\geq \left(\frac{1}{2} +\frac{2}{3(n-2)}\right)D(G).$$
\end{lemma}

\begin{proof}
If $n=3$, then $G$ should be the complete graph, $K_3$. In this case, $D(G)=1$ and by an easy
calculation $H(G)=\frac{3}{2}$ and both inequalities  hold.

If $n=4$, then since $K_4$ is not a quasi-tree graph, $D(G)>1$. Also since $G$ is
not a tree, $G\neq P_4$ and $D(G)=2$. Hence, $D(G)+\frac{5}{3}-\frac{n}{2}= \left(\frac{1}{2} +\frac{2}{3(n-2)}\right)D(G)=\frac{5}{3}$.
So both inequalities hold for $n=4$ (see Figure \ref{fign4}).

Suppose $n=5$ and $w$ is the vertex that $G-w$ is a tree. Since $G-w$ is a tree of order $4$,
$D(G-w)=3$ or $2$. Also since $G$ is a quasi-tree graph, it is not a complete graph, and hence, $D(G)=3$ or $2$.

If $D(G)=3$, then $ D(G)+\frac{5}{3}-\frac{n}{2}=\left(\frac{1}{2} +\frac{2}{3(n-2)}\right)D(G)=\frac{13}{6}$.
If $D(G)=2$, then $ D(G)+\frac{5}{3}-\frac{n}{2}=\frac{7}{6}$ and $ \left(\frac{1}{2} +\frac{2}{3(n-2)}\right)D(G)=\frac{13}{9}$.
All quasi-tree graphs of order $5$ and their harmonic indices are shown in Figure \ref{fign5}.
As it is seen, all  graphs hold both inequalities except when $G=U_{5,3}^{1,1}$.

Suppose $n=6$ and $w$ is the vertex that $G-w$ is a tree. So $G-w$ is one of $P_5,~K_{1,4}$ or $K_{1,3}^{+}$,
where $K_{1,3}^{+}$ is obtained  by attaching a new pendant vertex to a pendant vertex of $K_{1,3}$.

If $G-w=K_{1,4}$, then $D(G-w)=2$. Since $D(G)\leq D(G-w)$ and $G$ is not $K_6$, so $D(G)=2$, $D(G)+\frac{5}{3}-\frac{n}{2}= \frac{2}{3}$ and $\left(\frac{1}{2}+\frac{2}{3(n-2)}\right)D(G)= \frac{4}{3}$. On the other hand, for every edge $uv$ of $G-w$,
$\frac{2}{d_u+d_v}$ is at least $ \frac{2}{7}$, which $d_u$ and $d_v$
are the degree of $u$ and $v$ in $G$ respectively. So
\begin{align*}
H(G)\geq 4(\frac{2}{7})+\sum_{x\in N(w)} \frac{2}{d_x+d_w}\geq \frac{8}{7}+\frac{2}{5+d_w}+\frac{2(d_w-1)}{2+d_w}\geq
\frac{8}{7}+\frac{11}{14}>\frac{4}{3}.
\end{align*}
The second and third inequalities hold by Lemma \ref{casen6}.

If $G-w=P_5$, then the graph $G$ is one of the graphs shown in Figure \ref{fign61} which their harmonic indices are calculated.
Also $D(G)\leq 4$ and so $ D(G)+\frac{5}{3}-\frac{n}{2}\leq 2+\frac{2}{3}$ and
$\left(\frac{1}{2} +\frac{2}{3(n-2)}\right)D(G)\leq 2+\frac{2}{3}$. So for every
graph both inequalities hold, except when $G=U_{6,4}^{1,1}$. Also the equality holds when $G=V_{1,1}$.

If $G-w=K_{1,3}^{+}$, then the graph $G$ is one of the graphs shown in Figure \ref{fign62} which their harmonic indices are calculated.
Also $D(G)\leq 3$ and so $ D(G)+\frac{5}{3}-\frac{n}{2}\leq 1+\frac{2}{3}$ and
$\left(\frac{1}{2} +\frac{2}{3(n-2)}\right)D(G)\leq 2$. Obviousely, for every graph both inequalities hold.
\end{proof}
\begin{figure}[h]
\includegraphics[scale=.8]{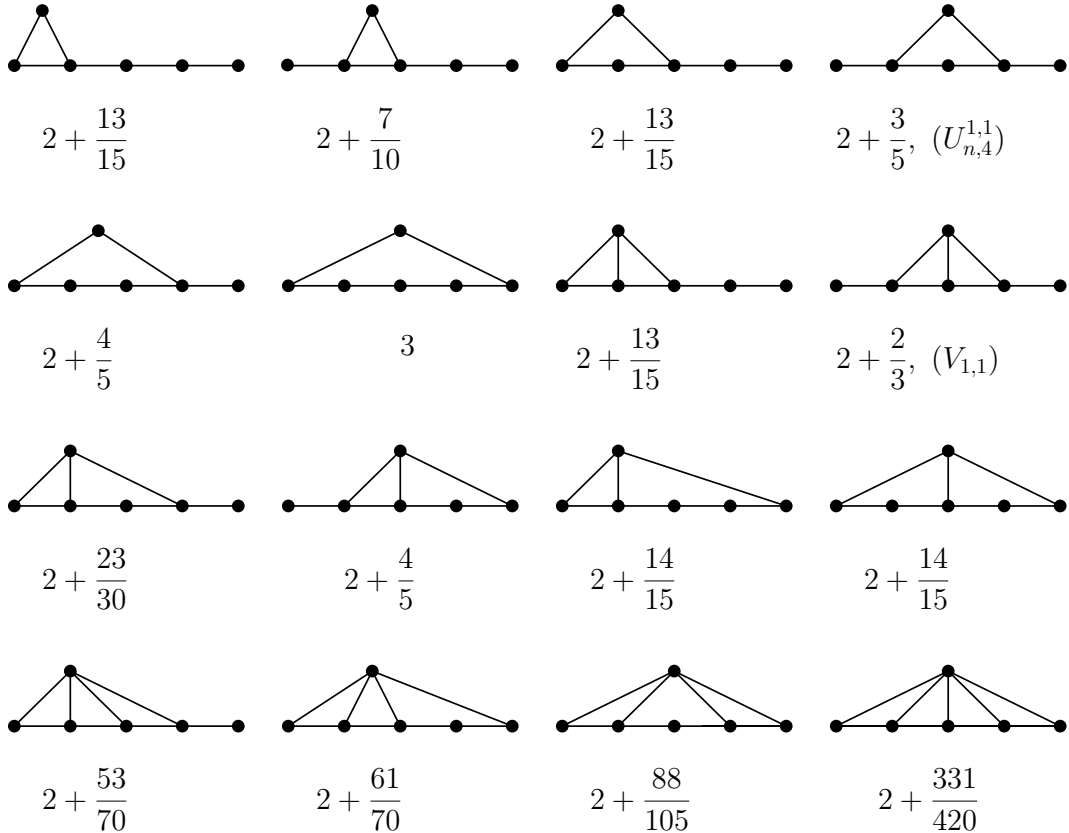}
\caption{Quasi-tree graphs of order $6$ obtained from $P_5$, and their harmonic indices.}\label{fign61}
\end{figure}
\begin{figure}[h]
\includegraphics[scale=.8]{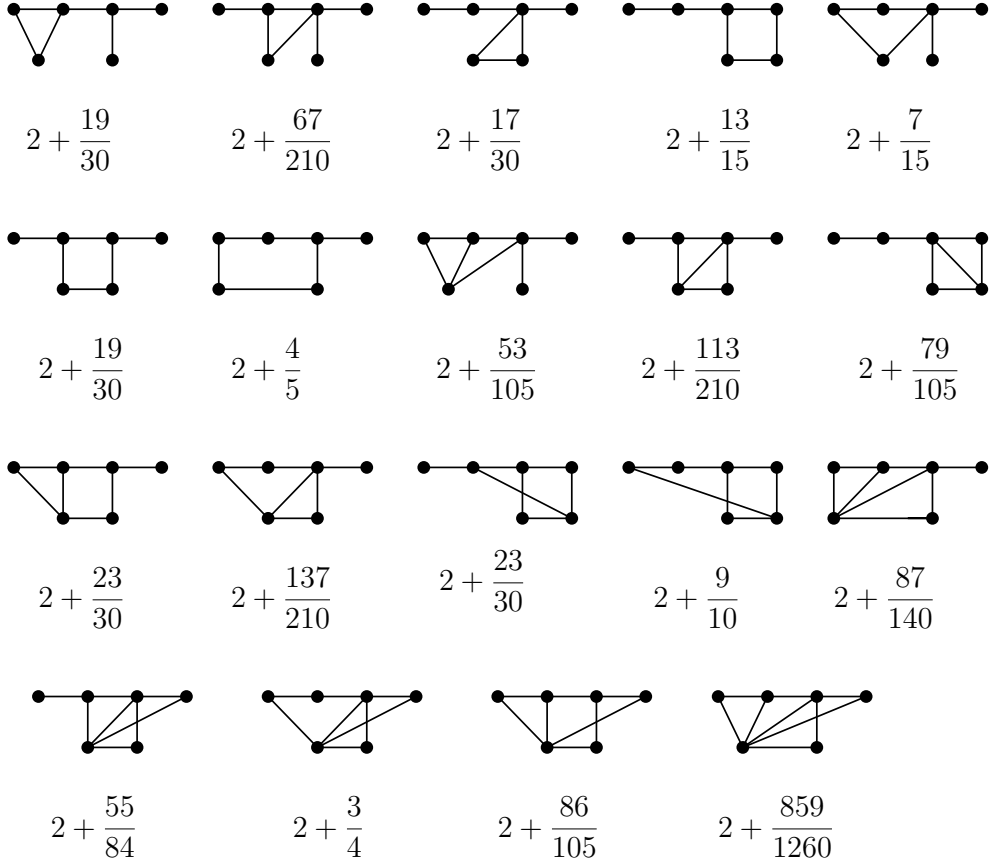}
\caption{Quasi-tree graphs of order $6$ obtained from $K_{1,3}^{*}$, and their harmonic indices.}\label{fign62}
\end{figure}
\begin{theorem}\label{th1}
Let $G\neq U_{6,4}^{1,1},U_{5,3}^{1,1}$ be a quasi-tree graph with $n\geq 3$ vertices. Then,
$ H(G)\geq D(G)+\frac{5}{3}-\frac{n}{2}$. The equality holds if and only if $G=V_{1,1}$ or $G=U_{n,4}^{1,n-5}$.
\end{theorem}
\begin{proof}
By induction on $n$, if $n\leq 6$, then Theorem \ref{th0} implies that the inequality is true, unless when $G=U_{6,4}^{1,1},U_{5,3}^{1,1}$.
Also by Theorem \ref{th0} the equality holds when $G=V_{1,1}$.

Let $G$ be a quasi-tree graph with $n\geq 7$ vertices. Suppose $w$ is a vertex of $G$ such
that $G-w$ is a tree. Let $P=u_0-u_1-\cdots-u_d$ be the diametrical path of $G$.
There are three cases as follows.
\item[\textbf{Case 1.}]
There exists $t\in G-P$ such that $d_t=1$.
Since $t$ is not in diametrical path of $G$ and $d_t=1$, $D(G-t)=D(G)$ and  $G-t$
is a quasi-tree graph too. Suppose $N(t)=\{r\}$ and $G-t\neq U_{6,4}^{1,1}$. Then, by induction hypothesis,
\begin{align*}
H(G)&=H(G-t)+\frac{2}{d_r+1}-\sum_{\substack{x\in N(r)\\x\neq t}}\frac{2}{d_r+d_x-1}+\sum_{\substack{x\in N(r)\\x\neq t}}\frac{2}{d_r+d_x}\\
&=H(G-t)+\frac{2}{d_r+1}-\sum_{\substack{x\in N(r)\\x\neq t}}\frac{2}{(d_r+d_x-1)(d_r+d_x)}\\
&\geq H(G-t)+\frac{2}{d_r+1}-\frac{2(d_r-1)}{d_r(d_r+1)}\geq (D(G)+\frac{5}{3}-\frac{n-1}{2})+\frac{2}{d_r(d_r+1)}\\&>D(G)+\frac{5}{3}-\frac{n}{2}.\\
\end{align*}
If $G-t=U_{6,4}^{1,1}$, then $G$ is one of the graphs  shown with their harmonic indices in Figure \ref{figt1}.
In this case, $D(G)=4$ and $ D(G)+\frac{5}{3}-\frac{n}{2}=\frac{65}{30}$. Hence, the inequality holds.
\item[\textbf{Case 2.}]
Every vertex of $G-P$ is of degree at least $2$ and there exists $t\in G-P-w$ such that $d_t=2$.
Similar to the previous case, $D(G)=D(G-t)$. Suppose $N(t)=\{r,s\}$. Note that $d_r,d_s\geq 0$, otherwise $\{r,s\}\cap P=\emptyset$ and hence, $t\in P$, a contradiction. Since $d_t=2$, it is
possible that $G-t$ be a tree. There exist three subcases as follows.
\item[\small\textbf{Subcase 2.1.}]
$G-t$ is not a tree and $r,s$ are not adjacent in $G$.
By the hypothesis, there exist at most two vertices of degree $1$ in $G$ which are $u_0$ and $u_d$. If $u_0,u_d\in N(r)$, then $D(G)=2$ and since $r,s$ are not adjacent, $d(s,u_0)\geq 3$, a contradiction. So $r$ has at most one neighbore of degree $1$. Same argument is valid for $s$ and so
\begin{align*}
H(G)&=H(G-t)+\frac{2}{2+d_r}+\frac{2}{2+d_s}\\&-2\sum_{\substack{x\in N(r)\\x\neq t}}\frac{1}{(d_r-1+d_x)(d_r+d_x)}-2\sum_{\substack{y\in N(s)\\y\neq t}}\frac{1}{(d_s-1+d_y)(d_y+d_s)}\\&
\geq D(G)+\frac{5}{3}+\frac{n-1}{2}+\frac{2}{2+d_r}+\frac{2}{2+d_s}\\&-\frac{2(d_r-2)}{(d_r+1)(d_r+2)}-\frac{2}{d_r(d_r+1)}-\frac{2(d_s-2)}{(d_s+1)(d_s+2)}-\frac{2}{d_s(d_s+1)}\\&
=D(G)+\frac{5}{3}-\frac{n}{2}+\frac{1}{2}+\frac{4(d_r-1)}{d_r(d_r+1)(2+d_r)}+\frac{4(d_s-1)}{d_s(d_s+1)(2+d_s)}>D(G)+\frac{5}{3}-\frac{n}{2}.
\end{align*}
%
If $G-t=U_{6,4}^{1,1}$, then $G$ is one of the graphs which are shown with their harmonic indices in Figure \ref{figt21}.
In this case, $D(G)=4$ and $ D(G)+\frac{5}{3}-\frac{n}{2}=\frac{65}{30}$.
\item[\small\textbf{Subcase 2.2.}]
$G-t$ is not a tree and $r,s$ are adjacent in $G$.
If $G-t\neq U_{6,4}^{1,1}$, then
\begin{align*}
H(G)&=H(G-t)+\frac{2}{d_r+2}+\frac{2}{d_s+2}+\frac{2}{d_r+d_s}-\frac{2}{d_r+d_s-2}\\&
-\sum_{\substack{x\in N(r)\\x\neq t,y}} \frac{2}{(d_x+d_r-1)(d_x+d_r)}-\sum_{\substack{y\in N(s)\\y\neq t,r}} \frac{2}{(d_y+d_s-1)(d_y+d_s)}\\
&\geq  H(G-t)+\frac{2}{d_r+2}+\frac{2}{d_s+2}-\frac{4}{(d_r+d_s)(d_r+d_s-2)}- \frac{2(d_r-2)}{d_r(d_r+1)}- \frac{2(d_s-2)}{d_s(d_s+1)}\\
&\geq H(G-t)-\frac{4}{(d_r+d_s)(d_r+d_s-2)}+\frac{2(d_r+4)}{d_r(d_r+1)(d_r+2)} +\frac{2(d_s+4)}{d_s(d_s+1)(d_s+2)}\\&> D(G)+\frac{5}{3}-\frac{n}{2}.
\end{align*}
Since $d_r, d_s\geq 2$, the last inequality is obtained from Lemma \ref{case2.2}.

If $G-t=U_{6,4}^{1,1}$, then $G$ is the graph which is shown with its harmonic index in Figure \ref{figt22}.
As Subcase 2.1, $ D(G)+\frac{5}{3}-\frac{n}{2}=\frac{65}{30}$ and the inequality holds.
\item[\small\textbf{Subcase 2.3.}]
$G-t$ is a tree.
Since $G-w$ is also a tree, by counting the number of edges and vertices, it is obtained  that $d_w=d_t=2$.
This means that $G$ is a unicyclic graph and as proved by \cite[Theorem~ 3.1.]{jerline2016conjecture}, $ H(G)\geq D(G)+\frac{5}{3}-\frac{n}{2}$,
with equality holds if $G=U_{n,4}^{1,n-5} $.
\item[\textbf{Case 3.}]
Every vertex of $G-\{u_0,u_1,\cdots,u_d\}$ is of degree at least $2$ and if
$t\in V(G)$ and $d_t=2$, then $t\in\{u_0,u_1,\cdots,u_d,w\}$.
Since $G-w$ is a tree, every pendant vertex of $G-w$ is in
$\{u_0,u_1,\cdots,u_d\}$. So $G-w$ is a path and $V(G)=\{u_0,u_1,\cdots,u_d\}\cup\{w\}$.
Also since $G$ is not a tree $d_w\geq 2$. If $d_w=2$, then $G$ is a unicyclic graph
and as proposed by \cite[Theorem~3.1.]{jerline2016conjecture}, $ H(G)\geq D(G)+\frac{5}{3}-\frac{n}{2}$, with equality holds if $G=U_{n,4}^{1,n-5}$.
So there exist two subcases as follows.
\item[\small\textbf{Subcase 3.1.}]
$d_w=3$
In this case, $G$ is one of the graphs in Table \ref{tabfig1}. In all cases, $ H(G)\geq D(G)+\frac{5}{3}-\frac{n}{2}$. As it is shown, the equality holds when $G=V_{1,1}$.
\item[\small\textbf{Subcase 3.2.}]
$d_w>3$
Suppose $\{u_h,u_i,u_j,u_k\}\subseteq N(w)$ such that $h<i<j<k$. If $k-h>2$, then the diametrical
path $u_1-\cdots-u_h-\cdots-u_k-\cdots-u_d$ is longer than the path $u_1-\cdots-u_h-w-u_k-\cdots,u_d$,
a contradiction. So $k-h\leq 2$, which is another contradiction. Hence, this case does not happen.
The inequality holds in all cases and equality holds if and only if $G=V_{1,1}$ or $U_{n,4}^{1,n-5}$.
\end{proof}
\begin{figure}[h]
\includegraphics[scale=.7]{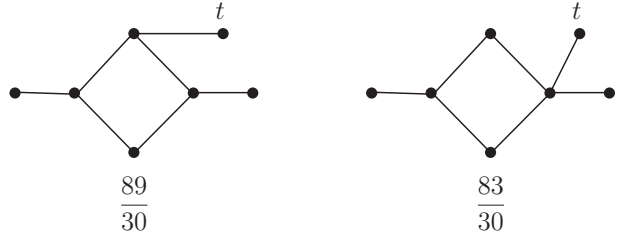}
\caption{The graphs related to Case 1 of Theorem \ref{th1} and Theorem \ref{th2}.}\label{figt1}
\end{figure}
\begin{figure}[h]
\includegraphics[scale=.7]{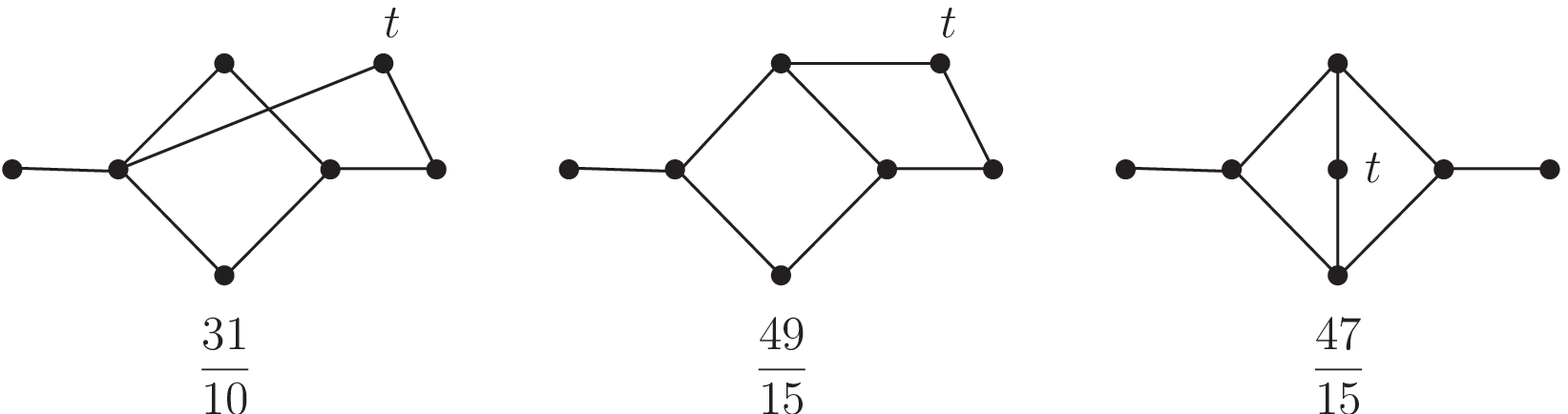}
\caption{The graphs related to Subcase 2.1 of Theorem \ref{th1} and Theorem \ref{th2}.}\label{figt21}
\end{figure}
\begin{figure}[h]
\includegraphics[scale=.7]{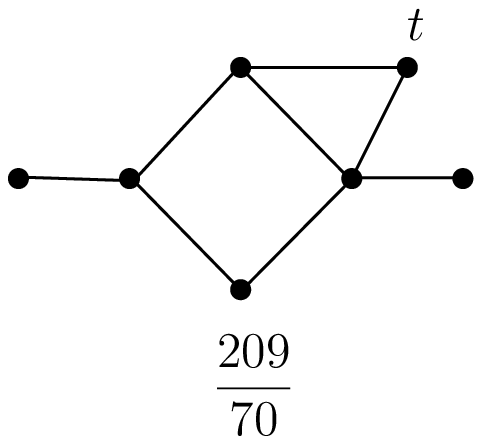}
\caption{The graphs related to Subcase 2.2 of Theorem \ref{th1} and Theorem \ref{th2}.}\label{figt22}
\end{figure}
\begin{table}[ht]
\begin{center}
\caption{All possibilities for Case 3 of Theorem \ref{th1}, \ref{th2}.}\label{tabfig1}
\begin{tabular}{|c|c|c|c|c|}
\hline

\multicolumn{2}{|c|}{Graph}&$H(G)$ & $D(G)+\frac{5}{3}-\frac{n}{2}$&$\left(\frac{1}{2}+\frac{2}{3(n-2)}\right) D(G)$\\
\hline\hline
\multirow{4}{*}{$V_{0,0}$}&\multirow{4}{*}{\includegraphics[height=1.3 cm]{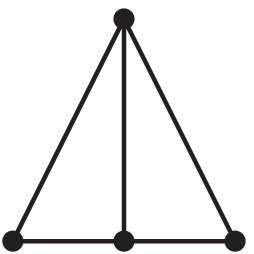}}&\multirow{4}
{*}{$\frac{29}{15}$}&\multirow{4}{*}{$\frac{25}{15}$}
&\multirow{4}{*}{$\frac{25}{15}$}\\
&&&&\\&&&&\\&&&&\\
\hline
\multirow{4}{*}{$V_{0,1}$}&\multirow{4}{*}{\includegraphics[height=1.3 cm]{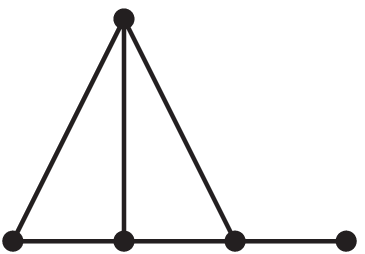}}&\multirow{4}{*}{
$\frac{23}{10}$}&\multirow{4}{*}{$\frac{13}{6}$}&
\multirow{4}{*}{$\frac{13}{6}$}\\ &&&&\\&&&&\\&&&&\\
\hline
\multirow{4}{*}{$V_{1,1}$}&\multirow{4}{*}{\includegraphics[height=1.3 cm]{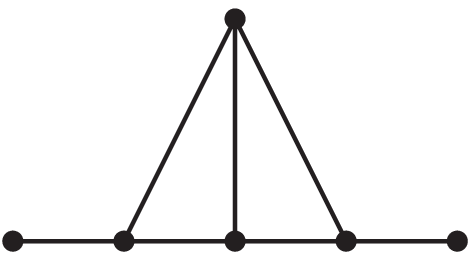}}&\multirow{4}{*}
{$\frac{8}{3}$}&\multirow{4}{*}{$\frac{8}{3}$}
&\multirow{4}{*}{$\frac{8}{3}$}\\&&&&\\&&&&\\&&&&\\
\hline
\multirow{4}{*}{$V_{0,r}$}&\multirow{4}{*}{\includegraphics[height=1.5 cm]{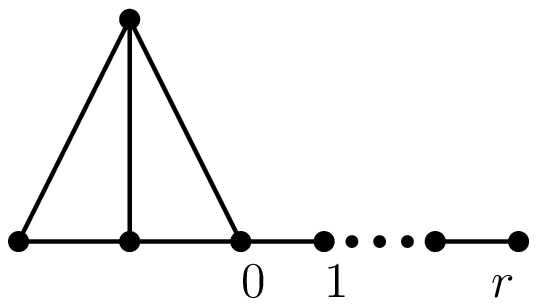}}&
\multirow{4}{*}{$ \frac{D(G)}{2}+\frac{13}{15}$}&\multirow{4}{*}{$\frac{D(G)}{2}+\frac{2}{3}$}&
\multirow{4}{*}{$\frac{D(G)}{2}+\frac{2}{3}$}\\
&&&&\\&&&&\\$r\geq 2$&&&&\\
\hline
\multirow{4}{*}{$V_{1,r}$}&\multirow{4}{*}{\includegraphics[height=1.5 cm]{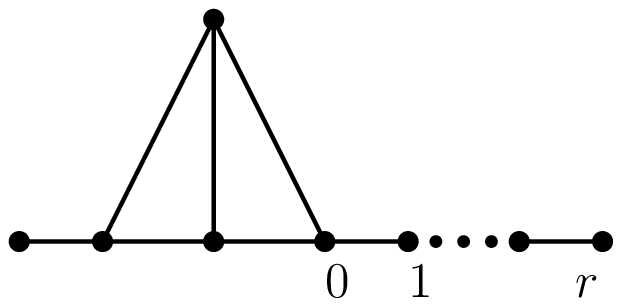}}&
\multirow{4}{*}{$ \frac{D(G)}{2}+\frac{11}{15}$}&\multirow{4}{*}{$\frac{D(G)}{2}+\frac{2}{3}$}&
\multirow{4}{*}{$\frac{D(G)}{2}+\frac{2}{3}$}\\
&&&&\\&&&&\\$r\geq 2$&&&&\\
\hline
\multirow{4}{*}{$V_{s,r}$}&\multirow{4}{*}{\includegraphics[height=1.5 cm]{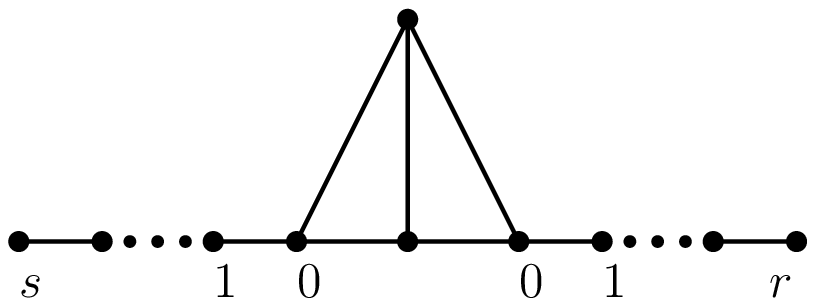}}&\multirow{4}{*}{$\frac{D(G)}{2}+\frac{4}{5}$}&\multirow{4}{*}{
$\frac{D(G)}{2}+\frac{2}{3}$}&\multirow{4}{*}{$\frac{D(G)}{2}+\frac{2}{3}$}\\&&&&\\&&&&\\$s,r\geq 2$&&&&\\
\hline
\end{tabular}
\end{center}
\end{table}

\begin{theorem}\label{th2}
Let $G\neq U_{6,4}^{1,1},~U_{5,3}^{1,1}$ be a quasi-tree graph with $n\geq 3$ vertices, then
$$H(G)\geq \left(\frac{1}{2} +\frac{2}{3(n-2)}\right)D(G).$$ The equality holds if and only if $G=V_{1,1}$ or $U_{n,4}^{1,n-5} $.
\end{theorem}
\begin{proof}
The proof is similar to the proof of Theorem \ref{th1}.
By induction on $n$, if $n\leq 6$, then Theorem \ref{th0} implies that the inequality holds unless when $G=U_{6,4}^{1,1},U_{5,3}^{1,1}$.

Let $G$ be a quasi-tree graph with $n\geq 7$ vertices. Suppose $w$ is a vertex
of $G$ such that $G-w$ is a tree. Let $P=u_0-u_1-\cdots-u_D(G)$ be the diametrical path of $G$. Since $G-w$
is a tree, $\delta(G)\leq 2$.
There exist three cases as follows.
\item[\textbf{Case 1.}]
There exists $t\in G-P$ such that $d_t=1$.
Since $t$ is not in diametrical path of $G$, $D(G)=D(G-t)$. Suppose $N(t)={r}$. So such as
Case 1 in the proof of Theorem \ref{th1}, if $G-t\neq U_{6,4}^{1,1}$, then by induction hypothesis,
\begin{align*}
H(G)\geq H(G-t)+\frac{2}{d_r(d_r+1)}>\left(\frac{1}{2} +\frac{2}{3(n-3)}\right)D(G)>\left(\frac{1}{2} +\frac{2}{3(n-2)}\right)D(G).
\end{align*}
If $G-t=U_{6,4}^{1,1}$, then $G$ is one of the graphs which is shown with their harmonic indices in
Figure \ref{figt1}. In this case, $D(G)=4$ and $ \left(\frac{1}{2} +\frac{2}{3(n-2)}\right)D(G)=\frac{38}{15}$.
Hence, the inequality holds.
\item[\textbf{Case 2.}]
Every vertex of $G-P$ is of degree at least $2$ and there exists a vertex $t\in G-P-w$ such that $d_t=2$.
As the previous case, $D(G)=D(G-t)$. Suppose $N(t)=\{r,s\}$. Since $d_t=2$ it
is possible that $G-t$ be a tree. There exist three subcases as follows.
\item[\small\textbf{Subcase 2.1.}]
$G-t$ is not a tree and $r,s$ are not adjacent in $G$.
 By the same argument as in the Subcase 2.1 of Theorem \ref{th1},
\begin{align*}
{H(G)}> H(G-t)+\frac{4(d_r-1)}{d_r(d_r+1)(2+d_r)}+\frac{4(d_s-1)}{d_s(d_s+1)(2+d_s)}>\left(\frac{1}{2} +\frac{2}{3(n-2)}\right)D(G).
\end{align*}
If $G-t=U_{6,4}^{1,1}$, then  $ \left(\frac{1}{2} +\frac{2}{3(n-2)}\right)D(G)=\frac{38}{15}$ and $G$ is one of the graphs which are shown with their harmonic indices in Figure \ref{figt21}.
\item[\small\textbf{Subcase 2.2.}]
$G-t$ is not a tree and $r,s$ are adjacent in $G$.
By Subcase 2.2 of Theorem \ref{th1},
\begin{align*}
H(G)&\geq H(G-t)+\frac{2(d_r+4)}{d_r(d_r+1)(d_r+2)}+\frac{2(d_s+4)}{d_s(d_s+1)(d_s+2)}-\frac{4}{(d_r+d_s)(d_s+d_r-2)}\\&>H(G-t).
\end{align*}
So by the induction hypothesis,
\begin{align*}
{H(G)}>\left(\frac{1}{2}+\frac{2}{3(n-3)}\right)D(G)>\left(\frac{1}{2}+\frac{2}{3(n-2)}\right)D(G).
\end{align*}
If $G-t=U_{6,4}^{1,1}$, then  $ \left(\frac{1}{2} +\frac{2}{3(n-2)}\right)D(G)=\frac{38}{15}$ and $G$ is the graph which is shown with its harmonic index
in Figure \ref{figt22} and the inequality holds.
\item[\small\textbf{Subcase 2.3.}]
$G-t$ is a tree.
By the same argument as in the proof of  Subcase 2.3 of Theorem \ref{th1}, $G$ is a unicyclic graph and by
\cite[Theorem 3.1.]{amalorpava2016harmonic}, $H(G)\geq \left(\frac{1}{2} +\frac{2}{3(n-2)}\right)D(G)$, with equality  when $G=U_{n,4}^{1,n-5}$
\item[\textbf{Case 3.}]
Every vertex of $G-P$ is of degree at least $2$ and if $t\in V(G)$ and $d_t=2$, then $t\in\{u_0,u_1,\cdots,u_d,w\}$.
By the same argument as in Case 3 of Theorem \ref{th1}, $G-w$ is a path and $V(G)=\{u_0,u_1,\cdots,u_d\}\cup\{w\}$.
Also since $G$ is not a tree, $d_w\geq 2$. If $d_w=2$, then $G$ is a unicyclic graph and
by \cite[Theorem 3.1.]{amalorpava2016harmonic}, $H(G)\geq \left(\frac{1}{2} +\frac{2}{3(n-2)}\right)D(G)$,
with equality  if $G=V_{1,1}$ or $G=U_{n,4}^{1,n-5}$. Also by the Subcase 3.2 of
Theorem \ref{th1}, $d_w\ngtr 3$. So $d_w=3$ and $G$ is one of the graphs in
Table \ref{tabfig1}. Obviousely $H(G)\geq \left(\frac{1}{2} +\frac{2}{3(n-2)}\right)D(G)$, for all of them.
As in Theorem \ref{th1}, the inequality holds in all cases and the equality be satisfy if
and only if $G=V_{1,1}$ or $U_{n,4}^{1,n-5}$.
\end{proof}

\bibliographystyle{plain}
\bibliography{harmonicrefrence}
\end{document}